\documentclass[preprint,12pt]{elsarticle}
\usepackage{mathtext}

\usepackage{floatflt}

\usepackage{graphicx}

\usepackage{epsfig}
\usepackage[cp1251]{inputenc}
\usepackage{amssymb}
\usepackage{amsthm}
\usepackage{cmap}

\usepackage{amsmath}% http://ctan.org/pkg/amsmath

\biboptions{sort&compress}

\newcommand{\sysn}{\left\{\begin{array}{rcl}}
\newcommand{\sysk}{\end{array}\right.}

\newcommand{\ingrw}[2]{\includegraphics[width=#1mm]{#2}}

\newtheorem{theorem}{Theorem}[section]

\theoremstyle{example}
\newtheorem{example}[theorem]{Example}

\newtheorem{proposition}[theorem]{Proposition}

\newtheorem{corollary}[theorem]{Corollary}
\theoremstyle{definition}
\newtheorem{definition}[theorem]{Definition}
\newtheorem{question}[theorem]{Question}
\journal{...}

\begin{document}

\title{Velichko's notions close to sequentially separability and their hereditary variants in $C_p$-theory}

\author{Alexander V. Osipov}

\address{Krasovskii Institute of Mathematics and Mechanics, \\ Ural Federal
 University, Yekaterinburg, Russia}

\ead{oab@list.ru}

%%%%%%%%%%%%%%%%%%% NOTATION %%%%%%%%%%%%%%%

\begin{abstract} A space $X$ is {\it sequentially separable} if there is a countable $S\subset X$ such that every point of $X$ is the limit of a sequence of points from $S$.
 In 2004, N.V. Velichko defined and investigated concepts close to sequentially separability: {\it $\sigma$-separability} and {\it $F$-separability}.
The aim of this paper is to study $\sigma$-separability and $F$-separability (and their hereditary variants) of the space $C_p(X)$  of
all real-valued continuous functions, defined on a Tychonoff space $X$, endowed with the pointwise convergence topology.  In particular, we proved that $\sigma$-separability coincides with sequential separability. Hereditary variants (hereditarily $\sigma$-separablity and hereditarily $F$-separablity) coincides with Fr\'{e}chet--Urysohn property in the class of cosmic spaces.

\end{abstract}


\maketitle %

\section{Introduction}

If $X$ is a topological space and $A\subseteq X$, then the sequential closure of $A$, denoted by $[A]_s$, is the set of all limits of sequences from $A$. A set is called {\it sequentially closed} if $A=[A]_s$. We denote by $[A]$ or $\overline{A}$ the closure of $A$ in $X$.

Let us recall that a topological space $X$ is

$\bullet$ {\it separable} if it contains a countable everywhere dense set;

$\bullet$ {\it sequentially separable} if $X$ contains a countable set $A$ sequentially dense in $X$, i.e. $X=[A]_s$;

$\bullet$ {\it hereditarily sequentially separable} if every subspace of $X$ is sequentially separable;

$\bullet$ {\it strongly sequentially separable} if it is separable and every dense countable subspace is sequentially dense;

$\bullet$ {\it Fr\'{e}chet-Urysohn} if for
every $A\subset X$ and $x\in \overline{A}$ there exists a sequence
in $A$ converging to $x$.

\medskip
It is obvious that every Fr\'{e}chet-Urysohn separable space is strongly sequentially separable.
It is clear that a hereditarily sequentially separable space is sequentially separable and a sequentially separable space is separable.

\medskip
In \cite{vel}, N.V. Velichko defined and investigated concepts close to sequentially separability such that $\sigma$-separability and $F$-separability.

\begin{definition} The {\it $\sigma$-closure} of a set $A$ in a space $X$ is the set $[A]_{\sigma}$ consisting of all points $x$ such that:

(a) $x\in [A]_s$;

(b) if a sequence $\langle x_n \rangle$ of points of $A$ converges to $x$ and $x\in [A\setminus \bigcup\{x_n: n\in \mathbb{N}\}]$, then there exists a sequence $\langle y_n \rangle$ of points of the set $A\setminus \bigcup\{x_n: n\in \mathbb{N}\}$ converging to $x$.
\end{definition}

\begin{definition}
A set $A$ is called {\it $\sigma$-dense} in $X$ if $X=[A]_{\sigma}$. A space $X$ is called {\it $\sigma$-separable} if it contains a countable $\sigma$-dense set.
\end{definition}

\begin{definition} The {\it $F$-closure} of a set $A$ in a space $X$ is the set $[A]_{F}$ consisting of all points $x$ such that if $x\in [A']$, $A'\subseteq A$, then there exists a sequence $\langle x_n \rangle$ of points of the set $A'$ converging to the point $x$.
\end{definition}

\begin{definition}
A set $A$ is called {\it $F$-dense} in $X$ if $X=[A]_{F}$. A space $X$ is called {\it $F$-separable} if it contains a countable $F$-dense set.
\end{definition}

A space $X$ is called {\it hereditarily $\sigma$-separable} ({\it hereditarily $F$-separable}) if every subspace of $X$ is $\sigma$-separable ($F$-separable).

A {\it network} in a space $X$ is defined as the family $\sigma$ of subsets of $X$ such that any open set of $X$ is the union of some subfamily of $\sigma$.
A space is a {\it cosmic space} if it has a countable network.

\newpage

We consider the following space notations:

$\bullet$ {\bf s} --- separable spaces;

$\bullet$ {\bf hs} --- hereditarily separable spaces;

$\bullet$ {\bf ss} --- sequentially separable spaces;

$\bullet$ {\bf hss} --- hereditarily sequentially separable spaces;

$\bullet$ {\bf sss} --- strongly sequentially separable spaces;

$\bullet$ {\bf FU} --- Fr\'{e}chet-Urysohn spaces;

$\bullet$ {\bf Fs} --- $F$-separable spaces;

$\bullet$ {\bf $\sigma$s} --- $\sigma$-separable spaces;

$\bullet$ {\bf h$\sigma$s} --- hereditarily $\sigma$-separable spaces;

$\bullet$ {\bf hFs} --- hereditarily $F$-separable spaces.

%$\bullet$ $nw(X)=\omega$ --- cosmic spaces.

\medskip

 Velichko proved that any cosmic space is hereditarily sequentially separable (Theorem 1 in \cite{vel}) and a space is hereditarily $\sigma$-separable if and only if it is hereditarily $F$-separable (Theorem 4 in \cite{vel}).

The following examples were constructed in \cite{vel}:

\medskip

(1) Separable sequential compact space that is non-sequential separable;

\medskip

(2) Sequential separable sequential compact space that is non-$\sigma$-separable;

\medskip

(3) Cosmic space that is non-$\sigma$-separable;

\medskip

(4) $F$-separable space that is non-Fr\'{e}chet-Urysohn;

\medskip

(5)  $\sigma$-separable space that is non-$F$-separable.

\medskip

 Thus we have the following the relationships between these classes of spaces and counterexamples (Diagram 1).

\begin{center}
\ingrw{90}{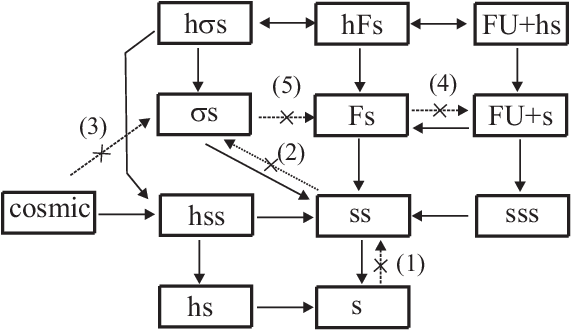}

\medskip

Diagram 1.

\end{center}

The set of positive integers is denoted by $\mathbb{N}$ and $\omega=\mathbb{N}\cup\{0\}$. Let $\mathbb{R}$ be the real line.
Let $X$ be a Tychonoff topological space, $C(X,\mathbb{R})$ be the
space of all  continuous functions on $X$ with values in
$\mathbb{R}$ and $\tau_p$ be the pointwise convergence topology.
Denote by $C_p(X)$ the topological space
$(C(X,\mathbb{R}), \tau_p)$.

 The symbol $\bf{0}$
stands for the constant function to $0$. A basic open neighborhood
of $\bf{0}$ in $C_p(X)$  is of the form $[F, (-\varepsilon, \varepsilon)]=\{f\in C_p(X)| \, f(F)\subset (-\varepsilon, \varepsilon)\}$, where $F$ is a finite subset of $X$ and $\varepsilon>0$.

The study of separability and sequential separability in function spaces can be found in many papers \cite{osi4,osi5,osi6,osi7,osi8}. In this paper we to study $\sigma$-separability and $F$-separability and their hereditary variants of the space $C_p(X)$. Unexpectedly, it turns out that $\sigma$-separability coincides with sequential separability, and hereditary analogues with strongly sequentially separability and hereditarily sequentially separability. In particular,  hereditary $\sigma$-separability of $C_p(X)$ coincides with Fr\'{e}chet--Urysohn property in the class of separable metrizable spaces $X$.

\section{Main results}

In \cite{Nob}, N. Noble proved that {\it a space $C_p(X)$ is separable if and only if $X$  has a coarser second countable topology}.

Thus, we assume that all spaces $X$ considered below are Tychonoff and has a coarser second countable topology.

\begin{definition} (Def. 2.2 in \cite{osi1}). A space $X$ has the {\it Velichko's property}, if there exists a continuous
bijection $f: X \rightarrow Y$ from a space $X$ on a separable metric space $Y$ such that $f(U)$ is an $F_{\sigma}$-set of $Y$ for
any cozero-set $U$ of $X$.
\end{definition}

In \cite{vel1} and \cite{GLM}, a criterion for the sequential separability of space $C_p(X)$ was obtained.

\begin{theorem}(\cite{vel1})\label{vel1} A space $C_p(X)$ is sequentially separable if and only if $X$ has the  Velichko's property.
\end{theorem}

\begin{theorem} The following assertions are equivalent:

\begin{enumerate}

\item $C_p(X)$ is sequentially separable;

\item $C_p(X)$ is $\sigma$-separable;

\item  $X$ has Velichko's property.

\end{enumerate}

\end{theorem}

\begin{proof} By Theorem \ref{vel1}, $(1)\Leftrightarrow(3)$.

 $(2)\Rightarrow(1)$. It is trivial.

$(1)\Rightarrow(2)$. Let $A=\{f_n: n\in \mathbb{N}\}$ be a countable sequentially dense set in $C_p(X)$.
Let $B=\bigcup\limits_{n\in \mathbb{N}} B_n$ where $B_n=\{f_n+r: r\in \mathbb{Q}\cap[0,\frac{1}{2^n}]\}$. It is clear that $B$ is a countable sequentially dense set in $C_p(X)$.

We claim that $B$ is $\sigma$-dense in $C_p(X)$. Let $f=\lim\limits_{i\rightarrow \infty} h_i$ where $\{h_i=f_{n_i}+r_i: i\in \mathbb{N}\}\subset B$ and $f\in [B\setminus \bigcup\{h_i: i\in \mathbb{N}\}]$.

Since $r_i\rightarrow 0$ ($i\rightarrow \infty$), $f=\lim\limits_{i\rightarrow \infty} f_{n_i}$. Let $q_i\in \mathbb{Q}\cap[0,\frac{1}{2^{n_i}}]$ for each $i\in \mathbb{N}$ and $\{r_i: i\in \mathbb{N}\}\bigcap \{q_i: i\in \mathbb{N}\}=\emptyset$.

(a) Suppose that there is $l\in \mathbb{N}$ such that $f\neq f_{n_i}$ for $i>l$.

Note that for every $i>l$  there is $s(i)\in\mathbb{N}$ such that   $s(i)\geq i$ and  $f_{n_i}+q_{s(i)}\neq f_{n_k}+r_k$ for any $k\in\mathbb{N}$. Otherwise the sequence  $(f_{n_i}+q_{s})$ ($s\in\overline{i,\infty})$  would be a subsequence of $(h_i)$ and it converged to the point $f$, but $(f_{n_i}+q_{s})$ converged to the point $f_{n_i}$.

Then the sequence $(g_i)=(f_{n_i}+q_{s(i)})$ such that $f=\lim\limits_{i\rightarrow \infty} g_i$ and  $\{g_i: i\in \mathbb{N}\}\subseteq B\setminus \bigcup\{h_i: i\in \mathbb{N}\}$.

(b) Suppose that $|\{i\in \mathbb{N}: f=f_{n_i}\}|=\omega$.

Then the sequence $(g_j)=(f_{n_j}+q_{j})=(f+q_{j})$ where $j\in \{i\in \mathbb{N}: f=f_{n_i}\}$ such that $f=\lim\limits_{j\rightarrow \infty} g_j$ and  $\{g_j\}\subseteq B\setminus \bigcup\{h_i: i\in \mathbb{N}\}$.

\end{proof}

A {\it cover} of a space is a family of {\it proper} subsets whose union is the entire space. For families $\mathcal{A}$ and $\mathcal{B}$ of covers of a space, the property that every cover in the family $\mathcal{A}$ has a subcover in the family $\mathcal{A}$ is denoted $\binom{\mathcal{A}}{\mathcal{B}}$. An $\omega$-cover is a cover such that each finite subset of the space is contained in some set from the cover. A $\gamma$-cover is an infinite cover such that each point of the space belongs to all but finitely many sets from the cover.

An {\it open cover} is a cover by open sets.  Given a space, let $\Omega$ and $\Gamma$ be the families of open $\omega$-covers and $\gamma$-covers, respectively. The property $\binom{\Omega}{\Gamma}$ is the celebrated $\gamma$-property of Gerlits and Nagy, who proved that a space has this property if and only if the space $C_p(X)$ is Fr\'{e}chet--Urysohn (\cite{GN}, Theorem 2).
A space $X$ is called {\it $\gamma$-space} if $X$ has property $\binom{\Omega}{\Gamma}$.

\newpage

\begin{theorem}\label{1} The following assertions are equivalent:

\begin{enumerate}

\item $C_p(X)$ is strongly sequentially separable and hereditarily sequentially separable;

\item $C_p(X)$ is hereditarily $F$-separable;

\item $C_p(X)$ is hereditarily $\sigma$-separable;

\item $C_p(X)$ is strongly sequentially separable and hereditarily separable;

\item $C_p(X)$ is Fr\'{e}chet--Urysohn and hereditarily separable;
%\item $X$  is projectively $\binom{\Omega}{\Gamma}$ and $X^n$ is hereditarily Lindel\"{o}f for every  $n\in \mathbb{N}$;

\item $X$ is a $\gamma$-space and $X^n$ is perfect normal for every $n\in \mathbb{N}$;

\item $X$ is a $\gamma$-space and $X^n$ is hereditarily Lindel\"{o}f for every $n\in \mathbb{N}$;

\item $X$ is a $\gamma$-space and $X^{\omega}$ is perfect normal;

\item $X$ is a $\gamma$-space and $X^{\omega}$ is hereditarily Lindel\"{o}f;

\item $X$ is a $\gamma$-space and $X^{\omega}$ is hereditarily normal;

\item $X$ is a $\gamma$-space and $X^{\omega}$ is hereditarily countably paracompact.

\end{enumerate}

\end{theorem}

\begin{proof} $(1)\Rightarrow(2)$. Let $A\subseteq C_p(X)$. Then $A$ is sequentially separable. Let $B\subseteq A$ be a countable sequentially dense subset of $A$ and $f\in [B']$ for $B'\subseteq B$. Since $C_p(X)$ is a homogeneous space, we can assume that $f={\bf 0}$.
By Theorem 7.5 in \cite{osi1}, $C_p(X)$ satisfies $S_1(\Omega^{\omega}_{\bf 0},\Gamma_{\bf 0})$, i.e. for each sequence $(C_n: n\in\mathbb{N})$ of countable subsets of $C_p(X)$ such that ${\bf 0}\in \bigcap [C_n]$ there is a sequence $(c_n : n\in \mathbb{N})$ such that for each $n$, $c_n\in C_n$, and $c_n\rightarrow {\bf 0}$ $(n\rightarrow \infty)$.

Let $C_n=B'$ for every $n\in \mathbb{N}$. Then there is $\{c_n : n\in \mathbb{N}\}\subseteq B'$ such that $c_n\rightarrow {\bf 0}$ $(n\rightarrow \infty)$.

$(2)\Leftrightarrow(3)$. By Theorem 4 in \cite{vel}.

$(3)\Rightarrow(4)$. Let $D$ be a countable dense subset of $C_p(X)$. Fix a point $f\in C_p(X)$. Then the set $D\cup \{f\}$ is $\sigma$-separable. Thus, there is $D'\subset D\cup \{f\}$ such that $f\in [D']_s$, i.e. there is $\{f_n: n\in \mathbb{N}\}\subseteq D'$ such that $f_n\rightarrow f$ $(n\rightarrow \infty)$. Suppose that there is $n'\in \mathbb{N}$ such that $f_n=f$ for $n>n'$. Note that $f\in [D'\setminus \{f\}]$. Then, by $\sigma$-separability of  $D\cup \{f\}$, there is $\{g_n: n\in \mathbb{N}\}\subseteq D'\setminus \{f\}$ such that $g_n\rightarrow f$ $(n\rightarrow \infty)$.
It remains to note that $D'\setminus \{f\}\subseteq D$.

$(4)\Rightarrow(5)$. By Velichko's Theorem (see  Theorem 2 in \cite{vel2} and Theorem II.5.33 in \cite{arch}), $X^n$ is hereditarily Lindel\"{o}f for every $n\in \mathbb{N}$. Then, $X$ is $\epsilon$-space, i.e.,
all finite powers of $X$ are Lindel\"{o}f (or, by Proposition in \cite{GN}, if every open $\omega$-cover of $X$ contains an at most countable $\omega$-subcover of $X$). By Theorem 7.5 in \cite{osi1} and Theorem 3.6 in \cite{koc}, $X$ has property $\binom{\Omega}{\Gamma}$, i.e. $X$ is a $\gamma$-space and, by Theorem 2 in \cite{GN},  $C_p(X)$ is Fr\'{e}chet--Urysohn.
%By Proposition 1 in \cite{vel2}, $C_p(X)^n$ is hereditarily separable for every  $n\in \mathbb{N}$.

$(5)\Rightarrow(6)$. By Theorem 2 in \cite{GN}, $X$ is a $\gamma$-space. By Velichko's Theorem (see  Theorem 2 in \cite{vel2} and Theorem II.5.33 in \cite{arch}), $X^n$ is hereditarily Lindel\"{o}f for every $n\in \mathbb{N}$. By Proposition 3.8.A.(b) in \cite{Eng}, $X^n$ is perfect normal for every $n\in \mathbb{N}$.

$(6)\Leftrightarrow(7)$. By Proposition 3.8.A.(b) in \cite{Eng}.

$(7)\Rightarrow(1)$. By Theorem 7.5 in \cite{osi1}, $C_p(X)$ satisfies $S_1(\Omega^{\omega}_{\bf 0},\Gamma_{\bf 0})$ and, by Theorem 2 in \cite{vel2}, $C_p(X)$ hereditarily separable. It follows that
$C_p(X)$ is strongly sequentially separable and hereditarily sequentially separable.

\medskip

From the following results, the remaining implications hold.

(a) A countable product space is hereditarily normal iff it is perfectly normal iff it is hereditarily
countably paracompact.

(b) A countable product space is perfect iff its all finite subproducts are perfect.

(c) The property $\binom{\Omega}{\Gamma}$ is preserved under taking finite or countable powers \cite{jmss}.

\end{proof}

Since the hereditarily separability of $C_p(X)$ is preserved under taking finite or countable powers (moreover, $C_p(X)$ is hereditarily separable iff $C_p(X^n)$ is hereditarily separable  iff $(C_p(X))^{\omega}$ is hereditarily separable (Proposition 1 in \cite{vel2})), we have the following corollary.

\begin{corollary} The following assertions are equivalent:

\begin{enumerate}

\item $C_p(X)$ is hereditarily $F$-separable;

\item $C_p(X)$ is hereditarily $\sigma$-separable;

\item $C_p(X^n)$ is hereditarily $\sigma$-separable for every $n\in \mathbb{N}$;

\item $C_p(X^n)$ is hereditarily $F$-separable for every $n\in \mathbb{N}$;

\item $C_p(X)^n$ is hereditarily $\sigma$-separable for every $n\in \mathbb{N}$;

\item $C_p(X)^n$ is hereditarily $F$-separable for every $n\in \mathbb{N}$;

\item $C_p(X)^{\omega}$ is hereditarily $\sigma$-separable;

\item $C_p(X)^{\omega}$ is hereditarily $F$-separable.

\end{enumerate}

\end{corollary}

By Theorem 1 in \cite{vel}, a cosmic space is hereditarily sequentially separable. Then we have the following corollary.

\begin{corollary}\label{cor} Let $X$ be a cosmic space. The following assertions are equivalent:

\begin{enumerate}

\item $C_p(X)$ is  strongly sequentially separable;

\item $C_p(X)$ is hereditarily $F$-separable;

\item $C_p(X)$ is hereditarily $\sigma$-separable;

\item $C_p(X)$ is Fr\'{e}chet--Urysohn;

\item  $X$ is a $\gamma$-space.

\end{enumerate}

\end{corollary}

It is consistent and independent for arbitrary $X$, that $C_p(X)$ is strongly sequentially separable if and only if $C_p(X)$ is Fr\'{e}chet--Urysohn. In fact it is consistent with ${\bf ZFC}$ that: $C_p(X)$ is strongly sequentially separable iff $X$ is countable iff $C_p(X)$ is separable and Fr\'{e}chet--Urysohn (Corollary 17 in \cite{GLM}).

\begin{proposition}\label{pr5} (Cons. ${\bf ZFC}$) $C_p(X)$ is separable and Fr\'{e}chet--Urysohn if and only if $C_p(X)$ is first countable.
\end{proposition}

\section{Examples}

Let us consider separating examples in $C_p$-theory.

\begin{example} There is a strongly sequentially separable space $C_p(X)$ such that it is not  hereditarily $F$-separable (hereditarily $\sigma$-separable).
\end{example}

In (\cite{osi2}, Corollary 11), assume that $\aleph_1<\mathfrak{p}$,  we constructed a strongly sequentially separable space $C_p(X)$ such that it is not  Fr\'{e}chet--Urysohn. By Theorem \ref{1}, the space $C_p(X)$ is not  hereditarily $F$-separable (hereditarily $\sigma$-separable).

\begin{example} There is a sequentially separable space $C_p(X)$ such that it is not  hereditarily separable.
\end{example}

The Sorgenfrey line $\mathbb{S}$ is the set $\mathbb{R}$ with the topology generated by the half-open intervals $[a, b)$, for $a, b\in \mathbb{R}$. M. Patrakeev proved (constructively) that there exists a countable sequentially dense subset in the space $C_p(\mathbb{S})$. Thus, (see also Velichko's criterion (Theorem  \ref{vel1})), the space $C_p(\mathbb{S})$ is sequential separable. The product $\mathbb{S}\times \mathbb{S}$ is not Lindel\"{o}f, hence, $C_p(X)$ is not hereditarily separable (Theorem 2 in \cite{vel2} and Theorem II.5.33 in \cite{arch}).

\begin{example} There is a separable space $C_p(X)$ such that it is not ($\sigma$-separable)  sequentially separable and it is not hereditarily separable.
\end{example}

It is enough to consider a discrete space $X$ of cardinality $\mathfrak{c}$. Then, the space $C_p(X)=\mathbb{R}^\mathfrak{c}$ is separable, but it is not ($\sigma$-separable)  sequentially separable (by Theorem  \ref{vel1}). Since $X$ is not Lindel\"{o}f, $\mathbb{R}^\mathfrak{c}$ is not hereditarily separable (Theorem 2 in \cite{vel2}).

\begin{example} There is a cosmic space $C_p(X)$ such that it is not (hereditarily $\sigma$-separable) hereditarily $F$-separable.
\end{example}

It is enough to consider any uncountable separable metric space $X$ which it is not $\gamma$-space. Let $X=[0,1]$. A space $C_p(Y)$ is cosmic if and only if $Y$ is cosmic (Theorem I.1.3 in \cite{arch}). Hence, $C_p([0,1])$ is cosmic, but $C_p([0,1])$ is not Fr\'{e}chet--Urysohn (Theorem II.3.4 in \cite{arch}). Thus, by Corollary \ref{cor}, $C_p([0,1])$ is not (hereditarily $\sigma$-separable) hereditarily $F$-separable.

\begin{example} There is a hereditarily separable space $C_p(X)$ such that it is not cosmic.
\end{example}

In \cite{mich}, assume that $(CH)$, Michael constructed a non-cosmic space $X$ such that $X^{\omega}$ is hereditarily Lindel\"{o}f.
Then $C_p(X)$ is non-cosmic (Theorem I.1.3 in \cite{arch}) and it is hereditarily separable (Theorem 2 in \cite{vel2}).

\medskip

Clear that any countable space is a $\gamma$-space. Then, by Theorem \ref{1}, we have that $C_p(X)$ is a hereditarily $\sigma$-separable (hereditarily $F$-separable) space for any countable space $X$.

\begin{example} There is a hereditarily $\sigma$-separable (hereditarily $F$-separable) space $C_p(X)$ such that it is not first countable.
\end{example}

It is enough to consider any uncountable separable metric $\gamma$-space.

\section{Open questions}

\begin{question} Suppose that $C_p(X)$ is a hereditarily separable space. Is it true that $C_p(X)$ is hereditarily sequentially separable?

\end{question}

\begin{question} Suppose that $C_p(X)$ is a hereditarily sequentially separable space. Is it true that  $C_p(X)$ is cosmic?

\end{question}

\begin{question} Suppose that $C_p(X)$ is a sequentially separable space. Is it true that $C_p(X)$ is $F$-separable?

\end{question}

\begin{question} Suppose that $C_p(X)$ is a $F$-separable space. Is it true that $C_p(X)$ is Fr\'{e}chet--Urysohn?

\end{question}

Note that the space $C_p(\mathbb{S})$ is sequentially separable space, but is not strongly sequentially separable (Example 3.2). Thus, $C_p(\mathbb{S})$ answers either Question 4.3 or Question 4.4.

\begin{question} Is there an example of a separable Fr\'{e}chet--Urysohn space $C_p(X)$ that is not hereditarily $F$-separable (Fr\'{e}chet--Urysohn and hereditarily separable)?

\end{question}

  By Proposition \ref{pr5}, this example depends on a model of set theory.

\medskip

The next question is a weakening of Question 4.1, but nevertheless has independent interest.

\begin{question} Suppose that $C_p(X)$ is a hereditarily separable space. Is it true that $C_p(X)$ is sequentially separable?
\end{question}

%\begin{proposition} There is a hereditarily $\sigma$-separable (hereditarily $F$-separable) non-cosmic space $C_p(X)$.

%\end{proposition}

%\begin{proof}

%\end{proof}

%\section{Preliminaries}

\medskip

%\section{Resulting Diagram}

 Thus we have the following the relationships between studied classes of $C_p$-spaces, counterexamples and open questions (Diagram 2).

\begin{center}

\ingrw{100}{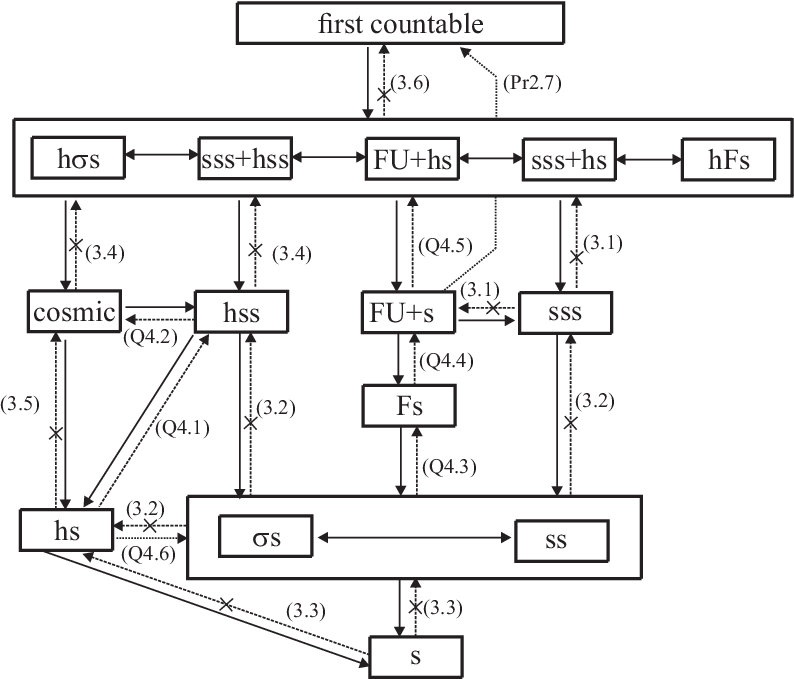}
\medskip

Diagram 2.

\end{center}

\end{document}